\documentclass[12pt]{amsart}
\usepackage{amsmath}
\usepackage{amsfonts}
\usepackage{amssymb}
\usepackage{graphicx}

\newtheorem{theorem}{Theorem}[section]
\newtheorem{lemma}[theorem]{Lemma}
\theoremstyle{plain}

\newtheorem{proposition}{Proposition}[section]

\numberwithin{equation}{section}

\begin{document}
\title[Base Change and Theta Correspondences]{Base Change and Theta
Correspondences for Supercuspidal Representations of \textit{SL}(2)}
\author{David Manderscheid}
\address{Office of the Dean\\
College of Arts and Sciences\\
1223 Oldfather Hall\\
University of Nebraska - Lincoln\\
Lincoln, NE 68588}
\email{dmanderscheid2@unl.edu}
\date{}

\begin{abstract}
Let $F$ be a $p$-adic field with $p$ odd. Quadratic base change and
theta-lifting are shown to be compatible for supercuspidal representations
of $SL(2,F)$. The argument involves the theory of types and the
lattice-model of the Weil representation.
\end{abstract}

\maketitle

%

%


\section*{Introduction}

Two commonly used methods to construct automorphic representations of
number-theoretic interest are theta-lifting and functoriality.
The extent to which these methods are compatible has proved a
fruitful avenue for research. In this paper we consider quadratic base
change for supercuspidal representations of $p$-adic $SL(2)$ as a theta-correspondence.

The result we obtain, that functoriality and theta-lifting are compatible in
this case while new and significant, is certainly not surprising given our
previous work \cite{M1, M2, M3, M4, M5}, that of Cognet \cite{C}, and the
literature, e.g., \cite{A, Au}. On the other hand the proof we give is 
surprisingly simple. It combines a number of known results and one new idea.
The idea is to blend various lattice models of the Weil
representation to give explicit models of the relevant types in the sense of
Bushnell-Kutzko. This new technique should prove of
value in a wide range of settings.

We now give a precise statement of the result and a description of the
method used to prove it. In particular, let $F$ be a $p$-adic field with $p$
odd, let $E/F$ be a quadratic extension, and let $G = G(F) = SL_2 (F)$. Let $%
O $ be the orthogonal group attended to a quadratic form on a four-dimensional vector space $V$ over $F$ 
and of Witt rank one over $F$. Then $O(V,F)=O(V)$ contains $PSL_2
(F)$ as a normal subgroup of index eight (see, e.g., Section 4). Now consider the reductive dual
pair $(G,O(V))$ and let $\chi$ be a non-trivial additive character of $F$. Let 
$w^{\infty} _\chi$ denote the smooth Weil representation associated to $(G,O(V))$ 
and $\chi$. Then the cocycle associated to $w^\infty _\chi$ splits on $%
G\cdot O(V)$ as a subgroup of the ambient symplectic group and the
restriction of $\omega^\infty _\chi$ to $G\cdot O (V)$ gives rise to a map $\theta
_\chi :\mathcal{R}_\chi (G)\rightarrow \mathcal{R}_\chi (O (V))$ where $\mathcal{R}_\chi
(G)$ is the set of irreducible smooth representations of $G$ that occur as
quotients of $\omega^\infty _\chi |_G$ and $\mathcal{R}_\chi (O (V))$ is defined
similarly. For $\pi$ a representation of $G$, let $L(\pi)$ denote its
associated $L$-packet and let $bc(L(\pi))$ be the $L$-packet for $%
PSL_2 (E)$ obtained from that of $L(\pi)$ by base change with respect to $E$%
. Then our main result is

\begin{theorem}
Let $\pi$ be an irreducible supercuspidal representation in $\mathcal{R}_\chi (G)$%
. Suppose that $bc(L(\pi))$ consists of supercuspidal
representations. Then $\theta _\chi (\pi)|_{{PSL_{2}} {(E)}}$ consists of
representations in $bc(L(\pi))$.
\end{theorem}

As commented previously, the proof involves bringing together a number
of known results and a relatively simple new idea. The new idea involves 
the theta correspondences attached to the pair $(G,O(V_1))$ where $V_1$ is
the vector space attached to a split ternary isotropic quadratic form. We
use these correspondences, as studied in \cite{M1, M5}, to explicitly
realize models for types in $PSL_2 (F)$ as a subgroup of $O(V_1)$. These
explicit models along with a compatible choice of lattice models for the
Weil representations attached to $(G,O(V_1))$ and $(G,O(V))$ then allow us
to complete the proof with a few straightforward calculations.

This paper is organized as follows. In the first section we set notation and
parametize the irreducible supercuspidal representations of $G$ and 
its nontrivial two-fold cover $\tilde{G}$ by types. In
the second section we parametize the irreducible supercuspidal
representation of $O(V_1)$. In the third section we recall relevant results
for the theta-correspondences attached to $(G,O(V_1))$. In the final section
we prove the main theorem. 

\section{Notation and Parameters}

In this section we establish notation and recall the parametrization, via
types, of the irreducible supercuspidal representations of $G$ and the 
genuine irreducible supercuspidal representations of $\tilde{G}$, with, in 
the case of $\tilde{G}$, genuine meaning that they do not factor to $G$. Since this
material is known (see, e.g., \cite{M1}) we will be brief in our discussion.

As in the introduction, let $F$ be a nonarchimedean local field and let $p$
denote the residual characteristic of $F$. For the remainder of this paper,
we assume $p\neq 2$. Let $\mathcal{O}=\mathcal{O}_{F},P=P_{F},\varpi =\varpi_{F}
,k=k_{F},q=q_{F}$ and $\left\vert\;\right\vert = \left\vert\;\right\vert_F$ 
denote, respectively, the ring of integers,
the prime ideal, a uniformizing parameter, the residue field, the order of
the residue field and the absolute value on $F$ normalized so that $%
|x|=q^{-\nu (x)}$ where $\nu =\nu _F$ denotes the order function on $F$. Let 
$U=U_F =\mathcal{O}_{F} ^\times$ and $U^n =U_{F} ^n =1+P_{F} ^n$ for $n$ a
positive integer. Now suppose that $K/F$ is a field extension. If $K/F$ is
Galois, let $\Gamma(K/F)$ denote the associated Galois group and if, in
addition, $[K:F]$ is finite, let $N_{K/F} =N$ denote the norm map and let $%
K^1 =K_{F} ^1$ be the group of norm one elements in $K^\times$.

For $G$ a group and $\sigma$ a representation of a subgroup $H$, let \linebreak 
$\mathrm{Ind}(G,H;\sigma)$ denote the representation of $G$ induced by $\sigma$ (form of
induction will either be specified or determined by context) and for $g$ in $%
G$, let $\sigma ^g$ denote the representation of $H^g =gHg^{-1}$ defined by $%
\sigma ^g (h)=\sigma (g^{-1} hg)$ for $h$ in $H^g$. If $J$ is a subgroup of $%
H$, let $\sigma |_J$ denote the restriction of $\sigma$ to $J$. Further, if $%
J\unlhd H$ and $\bar{\sigma}$ is a representation of $H/J$, then we view $%
\bar{\sigma}$ as a representation of $H$ via inflation. By a character we
mean a (not necessarily unitary) one-dimensional representation.

We now turn to the parametrization of the irreducible supercuspidal
representations of $G=SL_{2}(F)$ and $\tilde{G}$. Let $V$ be an $F$-vector
space of dimension two equipped with a skew-symmetric bilinear form $\langle
\; ,\,\rangle$. Then $G$ is the isometry group of $\langle \; ,\,\rangle$. Now
suppose that $K$ is a subfield of $A=A_{F}(V)=\mathrm{End}_{F}(V)$ properly
extending $F$ and hence quadratic over $F$. Then $K^1$ is contained in $G$
and we may view $V$ as an $K$-vector space. Since $V$ is one-dimensional
over $K$, there is a unique (up to equivalence) $\mathcal{O}_K$-lattice
chain $L$, say, in $V$. $L$ is also an $\mathcal{O}_F$-lattice chain of
length $e(K/F)$, the ramification degree of $K/F$.

Let $A=\mathcal{A}_F$ and $\mathcal{A}_K$ denote the principal orders in 
$A_{F}(V)$ and $A_{K}(V)$, respectively, associated to $L$. For $n$ an integer, 
let $\mathcal{P}_{F}^{n}=\{x\in A_{F}(V)\, | \, xL_{i}\subseteq 
L_{i+n}\mathrm{\, for\, all\,}i\}$; note that $\mathcal{P}_{F}^{0}=\mathcal{A}_F$. Similarly define 
$\mathcal{P}_{K}^n$. Recall that $\mathcal{P}_{F}=\mathcal{P}_{F}^1$ and $\mathcal{P}_{K}=\mathcal{P}_{K}^1$ are the Jacobson 
radicals of $\mathcal{A}_{F}$ and $\mathcal{A}_{K}$, respectively, and are invertible principal 
fractional ideals in $\mathcal{A}_{F}$ and $\mathcal{A}_{K}$ respectively. Recall also that 
$\mathcal{P}_{F}^{n}\cap A_{K}=\mathcal{P}_{K}^n$. We set $U(\mathcal{A}_{F})=U^{0}(\mathcal{A}_{F})=\{x\in G|x\in 
\mathcal{A}_{F}^{\times}\}$ and, for $n>0$, $U^{n}(\mathcal{A}_{F})=\{ x\in G\, | \, x-1\in \mathcal{P}_{F}^{n}\}$ 
and similarly for $K$. Then $U^{n}(\mathcal{A}_{F})\cap A_{K}=U^{n}(\mathcal{A}_{K})$ and if we 
identify $V$ with $K$ and thus $L$ with the fractional ideals on $K$ we get 
$U^{n}(\mathcal{A}_{K})=K_{n}^{1}=K^{1}\cap U_{K}^{n}$ for $n$ positive.

The usual trace map $tr=tr_{A/F}$ gives rise to a non-degenerate bilinear
form on $A_{F}$. We fix a nontrivial additive character $\chi$ of $F$. Let $%
k^{\prime }$ denote the subfield of $k$ of cardinality $p$ and let $\chi
_{k^{\prime }}$ be the additive character of $k^{\prime }$ with the property
that $\chi _{k^{\prime }}(1)=e^{2\pi i/p}$. Then we require that $\chi$
factors to the character $\chi _{k}$ of $k$ defined by $\chi _{k^{\prime
}}\circ Tr_{k/k^{\prime }}$. Note that $\chi$ has conductor $P_F$.

We say that $\alpha$ in $K$ is $K/F$-minimal if $K=F[\alpha ],Tr_{k/F}(\alpha )=0,\nu _{K}(\alpha )<0$,$(\nu _{K}(\alpha )$,$e(K/F))=1$, and if 
\begin{equation*}
\mathcal{O}_{F}[N_{K/K_{ur}}(\alpha )/\varpi_{F}^{\nu _{K}(\alpha )}]=%
\mathcal{O}_{K_{ur}}
\end{equation*}%
where $K_{ur}/F$ is the maximal unramified extension intermediate to $K/F$.
Note that this definition of minimal is more restrictive than the usual
definition, see, e.g., \cite{BK}, but suffices for our purposes.

Now let $\Lambda =\Lambda (K)$ denote the set of characters of $K^1$ and let 
$\Lambda _{1}$ denote the set of characters of $K^1$ trivial on $K^{1}\cap
U_{K}^1$. Let $\alpha$ in $K$ be $K/F$-minimal. Set $n=-\nu _{K}(\alpha)$ and 
$m^{\prime }=[(n+2)/2]$, where [\quad ] denotes the greatest integer
function. Define $\chi _\alpha$, a character on $U^{m^{\prime }}(\mathcal{A}_F)$, by 
$\chi _{\alpha}(1+x)=\chi (tr(\alpha x))$. Let $\Lambda _\alpha$ denote the
set of $\lambda$ in $\Lambda$ which agree with $\chi _\alpha$ upon
restriction to $U^{m^{\prime }}(\mathcal{A}_K)$. Then to each $\lambda$ in $\Lambda
_\alpha$ we may associate a character $\rho ^{\prime }(A,\lambda ,\alpha )$
of $K^{1}U^{m^{\prime }}(\mathcal{A}_F)$ in the usual manner.

Let $m=[(n+1)/2]$. If $n$ is odd, then $m=m^{\prime }$ and we set $\rho
(A,\lambda ,\alpha )=\rho ^{\prime }(A,\lambda ,\alpha )$. Then $\pi
(A,\lambda ,\alpha )=\mathrm{Ind}(G,K^{1}U^{m}(\mathcal{A}_{F});\,\rho (\mathcal{A},\lambda
,\alpha ))$ is an irreducible supercuspidal representation of $G$. If $n$ is
even (and thus $K/F$ is unramified), then $m=m^{\prime }-1$. Then there
exists a $q$-dimensional representation $\rho (\mathcal{A},\lambda ,\alpha )$ of 
$K^{1}U^{m}(\mathcal{A}_F )$ with the following properties.

\begin{enumerate}
\item The representations occurring in the decomposition of \linebreak $\mathrm{Ind}%
(K^{1}U^{m}(\mathcal{A}_{F}),K^{1}U^{m^{\prime }}(\mathcal{A}_{F})$; $\rho ^{\prime
}(\mathcal{A},\lambda ,\alpha ))$ are the representations $\rho \left( \mathcal{A},\lambda
^{\prime },\alpha \right) $ with $\lambda ^{\prime }\lambda ^{-1}$ in 
$\Lambda_1$, and $\lambda ^{\prime }\lambda ^{-1}\left( -1\right) =1$. These
representations occur with multiplicity two with the exception of $\rho
\left( \mathcal{A},\lambda ,\alpha \right) $ which occurs with multiplicity one.

\item $\pi \left( \mathcal{A},\lambda ,\alpha \right) =\mathrm{Ind}\left(
G,K^{1}U^{m}\left( \mathcal{A}_{F}\right) ;\rho \left( \mathcal{A},\lambda ,\alpha \right)
\right) $ is an irreducible representation of $G$.
\end{enumerate}

To construct the remaining irreducible supercuspidal representations of $G$,
suppose that $K/F$ is unramified and then let $\bar{V}=L_{0}/L_{1}$. The
symplectic form on $V$ gives rise to a symplectic form on $\bar{V}$, which
we also denote by $\langle \;,\,\rangle $. The associated isometry group is
isomorphic to $U ( \mathcal{A}_{F} ) / U^{1} ( \mathcal{A}_{F} ) $ and in turn $%
SL_{2} ( k ) $. As such, to each $\lambda$ in $\Lambda_{1}$ which
is not real-valued we associate a $\left( q-1\right) $-dimensional
representation $\rho \left( \mathcal{A},\lambda \right) $ of $U\left( \mathcal{A}_{F}\right) $
which is cuspidal as a representation of \linebreak $U\left( \mathcal{A}_{F}\right) /U^{1}\left(
\mathcal{A}_{F}\right) $ and with character $\theta _{\lambda }$ satisfying 
\begin{align*}
\theta _{\lambda }\left( -1\right) & =\left( q-1\right) \lambda \left( -1\right)  \\
\theta _{\lambda }\left( a\right) & =-\lambda \left( a\right) -\lambda
^{-1}\left( a \right) \quad\text{ for }a\text{ in }K^{1}/K_{1}^{1}\text{ but not
in }F
\end{align*}%
(see, e.g., \cite{Sp}). Then $\pi \left( \mathcal{}A,\lambda \right) =\mathrm{Ind}\left(
G,U\left( \mathcal{A}_{F}\right) ;\rho \left( \mathcal{A},\lambda \right) \right) $ is
irreducible and supercuspidal. To the signum character we associate two
representations $\pi \left( \mathcal{A},\pm \right) =\mathrm{Ind}\left( G,U\left(
\mathcal{A}_{F}\right) ,\rho \left( \mathcal{A},\pm \right) \right) $ where $\rho \left( \mathcal{A},+\right) 
$ and $\rho \left( \mathcal{A},-\right) $ are the two $\left( q-1\right) /2$ 
-dimensional cuspidal representations of $U\left( \mathcal{A}_{F}\right) /U^{1}\left(
\mathcal{A}_{F}\right) $. We fix these representations as follows. Let $\left(
X,Y\right) $ be a complete polarization of $\bar{V}$. Let $B=B_{Y}$
denote the Borel subgroup of $U\left( \mathcal{A}_{F}\right) /U^{1}\left( \mathcal{A}_{F}\right) 
$ preserving $Y$, let $N$ be the unipotent radical of $B$ and let $N_0$ denote the set of nontrivial $n$ in $N$ such that $\langle nv,v\rangle$ is a square for any vector $v$ in $X$. $N_0$ lies in a conjugacy class of  $U(\mathcal{A}_{F})/U^{1}(\mathcal{A}_{F})$. This class does not depend on the polarization but does depend on $\langle \;,\,\rangle$, which we have taken fixed. We specify $\rho (\mathcal{A},\pm )$ by requiring that their characters $\theta _\pm$ satisfy 
\begin{align*}
\theta _{\pm}(n)=(-1\pm\sigma )/2
\end{align*}
where $n$ is an element of the conjugacy class and 
\begin{align*}
\sigma =(-1)^{(q-1)/2}\sqrt{(-1)^{(q-1)/2}q}.
\end{align*}
Any future reference to representations of the form $\pi (\mathcal{A},\lambda )$ will include $\pi (\mathcal{A},\pm )$ and similarly for $\rho (\mathcal{A},\lambda )$.

\begin{theorem}\label{Theo1.1}\emph{(}see, e.g., \cite[Theorem 1.1]{M6}\emph{)}  
The representations of the form $\pi (\mathcal{A},\lambda )$ and $\pi (\mathcal{A},\lambda ,\alpha )$ exhaust the supercuspidal spectrum of $G$ and enjoy the following equivalences. 
\begin{enumerate}
\item A representation of the form $\pi (\mathcal{A},\lambda ,\alpha )$ is never equivalent to a representation of the form $\pi (\mathcal{A}',\lambda ')$.
\item $\pi (\mathcal{A},\lambda )$ and $\pi (\mathcal{A}',\lambda ')$ are equivalent if and only if there exists a $g$ in $G$ such that: 
\begin{enumerate}
\item $K'=K^g$
\item $\lambda '=\lambda ^g$ or $(\lambda ^{-1})^g$ if $\lambda$ is a character and $\lambda '=\lambda$ if $\lambda =\pm$.
\end{enumerate}
\item $\pi(\mathcal{A},\lambda ,\alpha )$ and $\pi (\mathcal{A}',\lambda ',\alpha ')$ are equivalent if and only if there exists a $g$ in $G$ such that:
\begin{enumerate}
\item $K'=K^g$
\item $\alpha '=\alpha ^{g}$ is in $(P_{K'})^{[-n/2]}$ where $n=-\nu_{\mathcal{A}}(\alpha')$
\item $\lambda '=\lambda ^{g}$.
\end{enumerate}
\end{enumerate}
\end{theorem}

Now we turn to $\tilde{G}$. We realize $\tilde{G}$ as the set of ordered pairs $(g,\xi )$ where $g$ is an element of $G$ and $\xi =\pm 1$ with multiplication given by $(g,\xi )(g',\xi ')=(gg',\beta (g,g')\xi \xi ')$ where $\beta$ is a non-trivial two-cycle with values in $\{\pm 1\}$ (see, e.g., \cite{M1}). Given a subgroup $H$ of $G$, let $\tilde{H}$ denote the inverse image of $H$ in $\tilde{G}$ under the map $(g,\xi )\mapsto g$ from $\tilde{G}$ to $G$. In an abuse of notation we write $g$ for the element $(g,1)$ in $\tilde{G}$. Since $p\neq 2$, $\beta$ splits on any compact subgroup. Then, after choosing a splitting, we may define $\tilde{\rho }(\mathcal{A},\lambda ,\alpha )$ on $\tilde{U}$ where $U=K^{1}U^{m}(\mathcal{A}_{F})$ by $\tilde{\rho}(\mathcal{A},\lambda ,\alpha )(g,\pm 1)=\pm \rho (\mathcal{A},\lambda ,\alpha )(g)$ and set $\tilde{\pi}(\mathcal{A},\lambda ,\alpha )=\mathrm{Ind}(\tilde{G},\tilde{U}; \, \rho (\mathcal{A},\lambda ,\alpha ))$. Define $\tilde{\pi}(\mathcal{A},\lambda)$ similarly.

\begin{theorem}\label{Theo1.2}\emph{(}see, e.g., \cite[Theorem 1.2]{M6}\emph{)} 
The representations $\tilde{\pi}(\mathcal{A},\lambda )$ and $\tilde{\pi}(\mathcal{A},\lambda ,\alpha )$ exhaust the set of genuine irreducible supercuspidal representations of $\tilde{G}$. They enjoy equivalences exactly as in Theorem \ref{Theo1.1}.
\end{theorem}

In what follows we often write, in an abuse of notation, $\pi (\mathcal{A},\lambda ,\alpha)$ and $\pi (\mathcal{A},\lambda )$ for $\tilde{\pi }(\mathcal{A},\lambda ,\alpha )$ and $\tilde{\pi} (\mathcal{A},\lambda )$, respectively. It will be clear from context whether the representation being considered is a representation of $G$ or $\tilde{G}$. We will also refer to $\mathcal{A}$ as the order parameter, $\lambda$ as the quadratic extension parameter, and $\alpha$ as the minimal element parameter.

\section{Supercuspidal representations of the split ternary orthogonal group}

In this section we parameterize the supercuspidal portion of the admissible dual of $O(F)$ where $O$ is the orthogonal group attached to a split ternary quadratic form over $F$. The material of this section is known (see, e.g., \cite{M6}) so once again we will be brief in our discussion. Recall that $O(F)\cong SO(F)\times \langle -I\rangle$ where $I$ is the identity matrix acting on the space. Further $SO(F)$ is isomorphic to $PGL_{2}(F)$.

Let $V_1$ be the subspace of $A_{F}(V)$ consisting of trace zero elements. Equip $V_1$ with the quadratic form $Q_{1}(A)=-\mathrm{det}A$. Then we map $GL_{2}(F)$ to $SO(V_{1})$ by sending $g$ in $GL_{2}(F)$ to the map given by conjugation by $g$. Then $PSL_{2}(F)$ can be identified with the commutator subgroup of $O(V_{1})$ and $PGL_{2}(F)$ with $SO(V_{1})$. Then $O(V_{1})/PSL_{2}(F)$ is abelian of type (2,2,2) and one can take generators to be $-I$ and $\beta_{i},i=1,2,$ where the spinor norm of $\beta _{1}$ is the class in $F^{\times}/(F^{\times})^2$ associated to the non-square unit and the spinor norm, given by the determinant map on $PGL_{2}(F)$, of $\beta _2$ is either of the classes associated to an element of order one.

The parameterization of the supercuspidal dual of $PSL_{2}(F)$ follows readily from that of $G$ by restricting to those $\lambda$ with $\lambda (-1)=1$. Note that $\mathrm{sgn}(-1)=1$ if and only if $K/F$ is unramified and $q\equiv 3$ ~ $\mathrm{mod}4$ (see, e.g., \cite[Lemma 2.2]{M6}).

For $\pi (\mathcal{A},\lambda ,\alpha )$ a representation of $PSL_{2}(F)$ let $\pi '(\mathcal{A},\lambda ,\alpha )$ = \linebreak $\mathrm{Ind}$($O(V_{1}), PSL_{2}(F);\pi (\mathcal{A},\lambda ,\alpha )$). Then $\pi '(\mathcal{A},\lambda ,\alpha )$ decomposes as the sum of four representations $\pi (\mathcal{A},\lambda ,\alpha ,\gamma_{1},\gamma_{2})$ where $\gamma_{i}, i=1,2$ are defined below. Define $\pi '(\mathcal{A},\lambda )$ similarly for $\pi (\mathcal{A},\lambda )$ a representation of $PSL_{2}(F)$. Then likewise $\pi '(\mathcal{A},\lambda )$ decomposes into four representations $\pi (\mathcal{A},\lambda ,\gamma_{1},\gamma_{2})$ with $\gamma_{i}$ as above. These representations are distinct if $\lambda\neq\pm$. If $\lambda =\pm$ then the four distinct representations obtained for $+$ are the same as those for $-$ and thus we denote these representations $\pi (\mathcal{A},\lambda ,\gamma_{1},\gamma_{2})$ with $\lambda =\mathrm{sgn}$.

$\gamma_2$ refers to the action of $-I$ and is $+1$ if the action is trivial and is $-1$ otherwise. The definition of $\gamma_1$ is more complicated and given below; see \cite{M6} for details.

First assume $q\equiv1\;\mathrm{mod}4$ or $K/F$ is ramified. Further consider representations of the form $\pi=\pi(A,\lambda,\alpha,\gamma_{1},\gamma_{2})$. We may take $\beta_{i},i=1$ or $2$, as $K/F$ is unramified or not, equal to $\alpha$ viewed as an element of $GL_{2}(F)$. Then there is, up to scalar, a unique vector $v$ in the space of $\pi$ transforming according to $\lambda$ on $PK^1$. Further $\pi(\beta_{i})v=cv$ for some $c=\pm 1$. Set $\gamma_{1}=c$.

Before turning to the case $q\equiv3\;\mathrm{mod}4$ and $K/F$ unramified for representations of the form $\pi(\mathcal{A},\lambda,\alpha,\gamma_{1},\gamma_{2})$, we consider $\gamma_1$ generally for representations of the form $\pi(\mathcal{A},\lambda,\gamma_{1},\gamma_{2})$. Let $\rho(\lambda')$ be the associated representation of $PGL_{2}(k_F)$ with $\lambda'$ a character of $K^\times$ such that $\lambda'|_{K^1}=\lambda$ and $\lambda'|_{F^\times}=1$.

If $q\equiv1\;\mathrm{mod}4$, let $\delta$ be such that $K=F[\delta]$ with $\mathrm{tr}\delta=0$ and $\delta$ representing the nontrivial coset in $K^{\times}/K^{1}F^{\times}$. Then set $\gamma_{1}=\lambda'(\delta)$.

If $q\equiv3\;\mathrm{mod}4$. Then, as in \cite{M6}, let $t=t(K)\geq2$ be such that $2^t$ is the highest power of $2$ dividing $q+1$. Then set $\gamma_{1}=\lambda'(\beta)$ where $\beta$ is a primitive $2^{t+1}$-root of unity in $K$ of norm $-1$. As noted in \cite{M6}, $\gamma_1$ is one of the two square roots of $\lambda(\beta^2)$. There is an ambiguity here coming from the choice of $\beta$. In all cases however where we will need to compare values of $\gamma_1$ attached to a $\beta$ we will be able to assume that the fields attached to the representations are equal and then we will be able to use the same $\beta$.

The final remaining case for $\gamma_1$ is for $\pi(\mathcal{A},\lambda,\alpha,\gamma_{1},\gamma_{2})$ with $q\equiv3\;\mathrm{mod}4$ and $K/F$ unramified. Let $\beta$ be as above a primitive $2^{t+1}$ root of unity of norm $-1$. Then with $v$ chosen as above $\pi(\beta)v=cv$ for a constant $c$ such that $c^{2}=\lambda(\beta^{2})$. Set $\gamma_{1}=c$. Once again there is ambiguity but we will be able to avoid it.

\begin{theorem}\label{theo2.1}\emph{(}see, e.g., \cite[Theorem 2.3]{M6}\emph{)} 
The representations \linebreak $\pi(\mathcal{A},\lambda,\alpha,\gamma_{1},\gamma_{2})$ and $\pi(\mathcal{A},\lambda,\gamma_{1},\gamma_{2})$ constructed above exhaust the supercuspidal spectrum of $O(V_{1})$. They enjoy the following equivalences. 
\begin{enumerate}
\item A representation of the form $\pi(\mathcal{A},\lambda,\alpha,\gamma_{1},\gamma_{2})$ is never equivalent to one of the form $\pi(\mathcal{A}',\lambda',\gamma_{1}',\gamma_{2}')$
\item Representations $\pi(\mathcal{A},\lambda,\gamma_{1},\gamma_{2})$ and $\pi(\mathcal{A}',\lambda',\gamma_{1}',\gamma_{2}')$ are equivalent if and only if $\gamma_{i}=\gamma_{i}'$ for $i=1,2$ and there exists a $g$ in $O(V)$ such that $K'=K^g$ and $\lambda'=\lambda^g$ or $(\lambda^{-1})^g$.
\item Representations $\pi(\mathcal{A},\lambda,\alpha,\gamma_{1},\gamma_{2})$ and $\pi(\mathcal{A}',\lambda',\alpha',\gamma_{1}',\gamma_{2}')$ are \linebreak equivalent if and only if $\gamma_{i}=\gamma_{i}'$ for $i=1,2$ and there exists a $g$ in $O(V_{1})$ such that $K'=K^g$ and 
\begin{enumerate}
\item $\alpha'-\alpha^g$ is in $(P_{K'})^{-[n/2]}$ where $n=\nu_{\mathcal{A}'}(\alpha')$;
\item $\lambda'=\lambda^g$
\end{enumerate}
\end{enumerate}
\end{theorem}

\section{A theta correspondence}

In this section we recall necessary facts concerning the theta- \linebreak correspondences attached to the reductive dual pair $(G,O(V_{1}))$. For details see \cite{M6}.

Recall the general settings of theta-correspondences for symplectic and orthogonal groups, temporarily suspending our previous notation. For $i=1,2$, let $W_i$ be a finite-dimensional vector space over $F$ equipped with a non-degenerate bilinear form $\langle\;,\,\rangle_i$, with $\langle\;,\,\rangle_1$ skew-symmetric and $\langle\;,\,\rangle_2$ symmetric. Equip $W=\mathrm{Hom}_{F}(W_{1},W_{2})$ with the non-degenerate skew-symmetric bilinear form $\langle\;,\,\rangle$ defined by $\langle w,w'\rangle=\linebreak \mathrm{tr}\,w\sigma(w')$ where $\sigma$ is the element of
\begin{align*}
\mathrm{Hom}_{F}(\mathrm{Hom}_{F}(W_{1},W_{2}),\mathrm{Hom}_{F}(W_{2},W_{1}))
\end{align*}
defined by $\langle Tw_{1},w_{2}\rangle_{2}=\langle w_{1},\sigma(T)w_{2}\rangle_{1}$ for all $w_{1}$, in $W_{1}$, and $w_2$ in $W_2$. Let $G_1$, $G_2$ and $G$ be the isometry groups of $\langle\;,\,\rangle_1$, $\langle\;,\,\rangle_2$, and $\langle\;,\,\rangle$ respectively. Identify $G_1$ and $G_2$ with subgroups of $G$ via their usual actions on $W$. Then $(G_{1},G_{2})$ is a reductive dual pair in $G$.

For $\psi$ a non-trivial additive character of $F$, let $\omega_{\psi}^\infty$ denote the (smooth) Weil representation of $\tilde{G}$, the non-trivial two-fold cover of $G$, associated to $\psi$. For $H$ a closed subgroup of $G$, let $\tilde{H}$ denote the inverse image of $H$ in $\tilde{G}$ and let $\mathcal{R}_{\psi}(\tilde{H})$ be the set of irreducible admissible representations of $H$ which occur as quotients of $\omega_{\psi}^{\infty}|_{\tilde{H}}$. Then $\mathcal{R}_{\psi}(\widetilde{G_{1},G_{2}})$ gives rise to theta-correspondences $\theta=\theta_{\psi}$:$\mathcal{R}_{\psi}(\tilde{G_{1}})\rightarrow\mathcal{R_{\psi}}(\tilde{G_{2}})$ and $\theta=\theta_{\psi}$:$\mathcal{R_{\psi}}(\tilde{G_{2}})\rightarrow\mathcal{R_{\psi}}(\tilde{G_{1}})$ with the direction of $\theta$ clear from context. Note that $\tilde{G_{2}}$ is always a trivial cover so we write, in an abuse of notation, $\mathcal{R_{\psi}}(G_{2})$ for $\mathcal{R_{\psi}}(\tilde{G_{2}})$ and consider the representations as representations of $G_2$.

In our proofs in section four of this paper we will use various lattice models of $\omega_\psi$. We will be brief in our discussion of the background for lattice models. For further details see, e.g., \cite{M2}. Recall that for a lattice $L$ in $W$, the dual of $L$ (with respect to $\psi$) is defined to be $L^{*}=\{w\in W|\psi(\langle w,\ell\rangle)=1$ for all $\ell$ in $L\}$. If $L$ is self dual then $\omega_{\psi}^{\infty}$ may be realized on the space $Y_{L}=Y$ of compactly supported functions on $W$ which satisfy 
\begin{align*}
f(w+\ell)=\chi(\langle w,\ell\rangle)f(w)
\end{align*}
for $w$ in $W$ and $\ell$ in $L$. For each $w$ in $W$ let $y_w$ denote the unique vector in $Y$ which is supported on $L-w$ with $y_{w}(-w)=1$.

If $L$ is not self-dual but satisfies
\begin{align*}
P_{F}L^{*}\!\subseteq L\!\subseteq L^*
\end{align*}
we can attach a non-self dual lattice model to $L$ as follows. View $\bar{L}=L^{*}/L$ as a symplectic vector space over $k_F$ with the inherited form. Let $(\rho_{\psi'},X)$ be a realization of the Heisenberg representation of $H(\bar{L})$, the Heisenberg attached to $\bar{L}$, with central character $\psi'$, 
where $\psi '$ is the character of $k_F$ coming from considering $\psi$ restricted to $P_{F}^{n}/P_{F}^{n+1}$ with $P_{F}^{n}$ the conductor of $\psi$. 
Let $e:W\rightarrow H(W)$ be the embedding as sets given by $w\mapsto (w,0)$ and $\gamma :H(W)\rightarrow W$ denote the map defined by $(w,t)\mapsto w$. Then let $J^*$ be the subgroup of $H(W)$ generated by $e(L^{*})$. Then we may inflate $\rho_{\psi'}$ to a representation of $J^*$. We may also then define a representation $\rho_L$ of $\gamma^{-1}(L^{*})$ on $X$ by $\rho_{L}(ah)(v)=\psi(a)\rho_{\psi'}(h)v$ where $a$ is in $Z(H(W))$, $h$ is in $J^*$ and $v$ is in $X$. Then $\mathrm{Ind}(H(W),\gamma^{-1}(L^{*});\rho_{L})$ realizes $\rho_\psi$, the Heisenberg representation of $H(W)$ associated to $\psi$.

Then the space $Y$ of $\omega_{\psi}^\infty$ is the space of compactly supported functions $f:W\rightarrow X$ such that 
\begin{align*}
f(w+a)=\psi((\langle w,a\rangle /2 )\rho_{L}(e(a))f(w)\,\textup{for}\,a\,\textup{in}\,L.
\end{align*}
Now let $y_{w,x}$ denote the function in $Y$ supported in $(-w+L^{*})$ such that $y_{w,x}(-w)=x$. Then the $y_{w,x}\;\mathrm{span}\;Y$.

Before considering specific dual pairs we recall two results concerning the action of compact subgroups in lattice models. In particular in a non self-dual model as above, suppose $M$ is a sublattice of $L$. Then $H_M =\{g\in G:(g-1)M^* \subset L^* \}$ is a subgroup of $G$. Further we have 

\begin{proposition}\label{prop3a}
If a function $f$ in $Y$ is supported on $M^*$, then 
\begin{align*}
w_\chi (h)f(w)=\rho_L (2c(h)w)\psi (\langle w,c(h)w\rangle )f(w)
\end{align*}
for $h$ in $H_M$ where $c(h)=(1-h)(1+h)^-1$ is the Cayley function of $h$.
\end{proposition}

On the other hand, assume that $L$ is self dual with $M$ in $L$ and set 
\begin{align*}
J_M =\{g\in G:(g-1)M^*\subseteq M\}
\end{align*}
and
\begin{align*}
H_M =\{g\in G|(g-1)M^*\subset L\}
\end{align*}
then $H_M$ and $J_M$ are subgroups of $G$ with $J_M$ normal in $H_M$ with abelian quotient and we have 

\begin{proposition}\label{prop3b} 
If $w$ is in $M^*$ then 
\begin{align*}
w_\chi (h)y_w =\psi (\langle hw,w\rangle /2)y_w
\end{align*}
for $w$ in $H_M$.
\end{proposition}

Generalities out of the way, we now proceed to a specific case.

Set $\psi = \chi, W_{1}=V$ and $W_{2}=V_{1}$ as in the previous sections. Let $\sigma =\sigma_{1}$.

\addtocounter{theorem}{2}
\begin{theorem}\label{theo3.3}\emph{(}see, e.g., \cite[Theorem 3.3]{M6}\emph{)}
The following hold:
\begin{enumerate}
\item Let $K=F[\alpha]/F$ be unramified with $\alpha$ $K/F$-minimal. If $\mathcal{O}_{K}$ is self-dual and $\nu_{K}(\alpha)$ is odd, then $\pi(\mathcal{A},\lambda ,\alpha)$ is in $\mathcal{R}_{\chi}(\tilde{G})$ for all $\lambda$. If $\mathcal{O}_{K}^{*}=P_{K}^{-1}$ and $\nu_{K}(\alpha)$ is even, then $\pi(A,\lambda ,\alpha)$ is in $\mathcal{R}_{\chi}(\tilde{G})$ for all $\lambda$ as is $\pi(\mathcal{A},\lambda)$ with the exception of $\lambda =\mathrm{sgn}(-A)$ where 
\begin{align*}
A=(-1)^{[(f+1)/2](p-1)/2}(-1)^{f+1}(-1)^{(p^{2}-1)f/8}
\end{align*}
with $f=\mathrm{log}_{p}(q)$. Finally no other supercuspidal representations attached to unramified $K/F$ are in $\mathcal{R}_\chi (\tilde{G})$.
\item Let $K=F[\alpha]/F$ be ramified with $\alpha$ $K/F$-minimal. Then $\pi(A,\lambda ,\alpha)$ is in $\mathcal{R}_\chi (\tilde{G})$ if and only if there exists a $w$ in $W$ such that $\varpi_{F}\sigma_{1}(w)w=-2\alpha$.
\item Assume that $K/F$ is ramified or that $q=1\mathrm{mod}4$. If $\pi(A,\lambda ,\alpha)$ is in $\mathcal{R}_\chi(\tilde{G})$ then the corresponding representation in $\mathcal{R}_\chi(O(V_{1}))$ is $\pi(\mathcal{A},\lambda^{2},\alpha/2,\,\lambda(-1),\lambda(-1))$. If $\pi(\mathcal{A},\lambda)$ is in $\mathcal{R}_\chi(\tilde{G})$ with $\lambda\neq\pm$, then the corresponding representation in $\mathcal{R}_\chi(O(V_{1}))$ is $\pi(\mathcal{A},\lambda^{2},\lambda(-1),\lambda(-1))$. The occurring $\pi(\mathcal{A},\pm)$ corresponds to a non-supercuspidal representation of $O(V_{1})$.
\item Assume that $K/F$ is unramified and that $q\equiv3\mathrm{mod}4$. Then if $\pi(\mathcal{A},\lambda ,\alpha)$ is in $\mathcal{R}_{\chi}(\tilde{G})$, then the corresponding representation in $\mathcal{R}_{\chi}(O(V_{1}))$ is $\pi(\mathcal{A},\lambda^{2},\gamma_{1},\lambda(-1))$ for some $\gamma_1$. If $\pi(A,\lambda)$ is in $\mathcal{R}_\chi (\tilde{G})$ with $\lambda\neq\pm$ then the corresponding representation $\mathcal{R}_{\chi}(O(V_{1}))$ is $\pi(\mathcal{A},\lambda^{2},\gamma_{1},\lambda(-1))$ for some $\gamma_1$. The occurring $\pi(\mathcal{A},\pm)$ pairs with a non-supercuspidal representation of $O(V_{1})$.
\end{enumerate}
\end{theorem}

\section{Explicit realization of types via lattice models and the theorem}

Let $E/F$ be a quadratic extension and set $V=\{A\in M_{2}(E)|\bar{A}^{t}=A\}$ where $\bar{A}^t$ denotes the conjugate transpose of $A$ with conjugation entrywise by the Galois action of $E/F$. Then the negative of the determinant map defines a Witt rank one quadratic form over $F$. Let $H=O(V)$.

Recall the structure of $H$. First let $H'=GO(V)$. Define $\Psi:GL_{2}(E)\times F^{\times}\rightarrow\mathrm{End}_{F}(V)$ by 
\begin{align*}
\Psi(g,u)A=\left(
\begin{array}{cc}
u & o \\
o & u
\end{array}
\right)
gA\bar{g}^t
\end{align*}
Then $H'\cong\mathrm{Im}\Psi\rtimes\langle\sigma\rangle$ where $\sigma$ is the element of $H$ corresponding to conjugation. Now consider the restriction of $\Psi$ to those elements of the form $(g,a)$ with $N(\mathrm{det}g)a^{2}=1$, call this group $H_1$. Then $\Psi(H_{1})\cong SO(V)$ and we have $H\cong\Psi(H_{1})\rtimes\langle\sigma\rangle$. Then we identify $PSL_{2}(E)$ as a subgroup of $H$ via $g\mapsto\Psi(g,1)$. Note that $H/PSL_{2}(E)$ is abelian of type (2,2,2) and we can take generators to be $\sigma$ and $\beta_{i},i=1,2$ where the spinor norm of $\beta_{1}$ is the class in $F^{x}/(F^{x})^2$ associated to a nonsquare unit and the spinor of $\beta_{2}$ is associated to an element of order one. Note further that the $\beta_{i}$ can be taken to be of the form $(g,a)$ with $g$ in $GL_{2}(F)$. We can further take $g$ to be zero on the diagonal and the $\Psi(g,1)$ becomes conjugation by $g$.

We will be interested in the supercuspidal representations of $H$ in $\mathcal{R}_\chi(H)$ which are summands of $\mathrm{Ind}(H,PSL_{2}(E),\Pi)$ for $\Pi$ a supercuspidal representation of $PSL_{2}(E)$ that arises from base change for $E/F$. For $\pi$ a supercuspidal representation of $SL_{2}(F)$, let $L(\pi)$ denote its associated $L$-packet and let $bc(L(\pi))$ denote the associated $L$-packet for $PSL_{2}(E)$ obtained by base change. 

Suppose $\pi$ is a supercuspidal representation of $SL_{2}(F)$ such that $bc(L(\pi))$ consists of supercuspidal representations. Then $\pi$ is  of the form $\pi(A,\lambda ,\alpha)$ with $E[\alpha]/F$ biquadratic or $\pi$ is of the form $\pi(A,\lambda)$ with $E/F$ ramified and $\lambda\neq\pm$. In the case of $\pi(A,\lambda)$ let $K$ be the unramified extension of $F$ on which $\lambda$ is defined. Then let $bc(\mathcal{A})$ be the order associated to the lattice chain associated to $E[\alpha]/F$ in the former case and $KE$ in the latter case.

\begin{lemma}\label{Lem4.1}
Let $\pi$ be an irreducible supercuspidal representation of $SL_{2}(F)$ such that $bc(L(\pi))$ consists of supercuspidal representations. 

\begin{enumerate}
\item If $\pi=\pi(\mathcal{A},\lambda ,\alpha)$ then $bc(L(\pi))$ contains \newline $\pi(bc(\mathcal{A})),\lambda_{\,^{\;\circ}} N_{E[\alpha]/F[\alpha]},\alpha/2)$
\item If $\pi=\pi(\mathcal{A},\lambda)$ then $bc(L(\pi))$ contains $\pi(bc(\mathcal{A}),\lambda_{\,^{\;\circ}} N_{LE/L})$
\end{enumerate}
\end{lemma}

\begin{proof}
The first case follows from Definition 2.1.4 and Theorem 4.10.1 of \cite{P} and the equality that 
\begin{eqnarray}
\chi_{F}(tr\alpha Tr_{E/F}(x-1))&=&\chi_{F}(tr\circ Tr_{E/F}((\alpha /_{2})(x-1)))\nonumber \\
&=&\chi_{F}(Tr_{E/F}(tr((\alpha/2)(x-1)))\nonumber \\
&=&\chi_{E}(tr((\alpha/2)(x-1)))\nonumber
\end{eqnarray}
The second case follows from \cite{Sh}.
\end{proof}

We now state and prove the main result.

\setcounter{theorem}{1}
\begin{theorem}\label{Theo4.2}
Let $\pi$ be a supercuspidal representation of $G$ occurring in $\mathcal{R}_\chi(G)$. Suppose that $bc(L(\pi))$ consists of supercuspidal representations. Then $\theta_{\chi}(\pi)|_{PSL_{2}(E)}$ consists of representations in $bc(L(\pi))$.
\end{theorem}

\begin{proof}
Let $\rho$ be a representation of the form $\rho(A,\lambda)$ or $\rho(A,\lambda ,\alpha)$ of a compact open subgroup $H$, as above, so that $\pi\cong\mathrm{Ind}(G,H;\rho)$. Then the cocycle defining $\tilde{G}$ splits on $\tilde{H}$ so $\rho$ can be extended to a representation $\tilde{\rho}$ on the trivial covering $\tilde{H}$ by setting $\tilde{\rho}((h,z))=z\rho(h)$ where $h$ is in $H$ and $z=\pm1$. Let $\tilde{\pi}=\mathrm{Ind}(\tilde{G},\tilde{H};\tilde{\rho})$ be the corresponding irreducible genuine representation of $\tilde{G}$. 

Let $D$ denote the non-split quaternion algebra over $F$ and let $D^0$ denote the (reduced) trace zero elements in $D$. Then the negative of the reduced norm gives rise to an anisotropic quadratic form on $D^0$. Let $V_{2}=D^0$. Then it is a result of Waldspurger \cite{W} that $\tilde{\pi}$ has a non-zero theta lift to one and only one of $O(V_i)$, $i=1,2$, for $\chi$.

First assume this occurs for $i=1$. Then note that for $g$ in $SL_{2}(F)$
\begin{eqnarray}
gA\overline{g}^{t}=gAwg^{-1}w^{-1}
\label{Eq1}
\end{eqnarray}
for $A$ in $V$ where $w=\left(\!\!
\begin{array}{rr}
0&1\\
-1&0
\end{array}
\right)$. 
In particular this holds for $A$ with the property that $\overline{A}= A$; let $V'_{1} =\{Aw|A\in V \textrm{with }\overline{A} =A\}$. Then one checks that $V'_{1}\cong V_{1}$. Further using (\ref{Eq1}) one checks that $PSL_{2}(F)$ as a subgroup of $O(V_{1})$ is equal to $PSL_{2}(F)$ on a subgroup of $PSL_{2}(E)$ when $O(V_{1})$ is viewed as a subgroup of $O(V)$.

Now suppose $\tilde{\pi}$ is of the form $\pi(A,\lambda)$. For the associated $L$-packet to remain supercuspidal we must have that $E/F$ is ramified. Then viewing $\lambda$ as a character of $K^1$ where $K/F$ is unramified we get that the representation corresponding to $\tilde{\pi}$ in $\mathcal{R}_{\chi}(O(V_{1}))$ is of the form $\pi(\mathcal{A},\lambda^{2},\gamma_{1},\gamma_{2})$ by Theorem \ref{theo3.3}. Then the result follows from Lemma \ref{Lem4.1}, our choice of $V_{1}'$, the equality $\lambda_{\,^{\;\circ}}N_{KE/K}(e)=\lambda^{2}(e)$ for $e$ in $E^1$, and the equality $SL_{2}(k_{F})=SL_{2}(k_{E})$.

Now suppose $\pi$ is of the form $\pi(\mathcal{A},\lambda ,\alpha)$. Then the representation corresponding to $\tilde{\pi}$ in $\mathcal{R}_{\chi}(O(V_{1}))$ is of the form $\pi(\mathcal{A},\lambda^{2},\alpha /2,\gamma_{1},\gamma_{2})$ by Theorem \ref{theo3.3}.  Let the subgroups of $\tilde{G}$ and $O(V_{1})$ where $\tilde{\rho}(\mathcal{A},\lambda ,\alpha)$ and $\rho(\mathcal{A},\lambda^{2},\alpha /2)$ are defined be $\tilde{H}_{1}$ and $H_2$. Consider the representation $\rho'=\tilde{\rho}(\mathcal{A},\lambda ,\alpha)\otimes\rho(\mathcal{A},\lambda^{2},\alpha /2)$ of $\tilde{H}_{1}H_{2}$ inside the ambient symplectic group for the relevant Weil representation. Examining the proof of Theorem 3.5 in \cite{M1} we see that the space of $\rho'$ can be taken to consist of functions with support in $W_{1}\otimes V_{1}$ in a lattice model, where we recall that $W_{1}$ is the symplectic space for $SL_{2}(F)$. We may take the lattice to be of the form $M\otimes L_{1}$ where $M$ is a lattice in $W_1$ and $L_1$ is a lattice in $V_1$. We then realize the Weil representation $\omega_{\chi}^{\infty}$ for $Sp(W)$ in a lattice model attached to a lattice of the form $M\otimes(\phi(L_{1})\oplus L_{2})$ where \linebreak $\phi :V_{1}\rightarrow V_{1}'$ is the isomorphism indicated above and $L_2$ is a lattice in $V_{1}''$. Then our statements about the defining order parameter and minimal element parameter follow from Proposition \ref{prop3a} since $\phi(L_{1})$ and $L_2$ are orthogonal. The statement for the parameter on the compact torus follows from the fact the action of $E^1$ is trivial. 

Now suppose $i=2$. Then the argument proceeds as above but we alter the isotropic form attached to $V$ by multiplication by a scalar so that $\tilde{\pi}$ does occur.
\end{proof}

\end{document}